\documentclass[leqno]{article}
\usepackage{geometry}
\usepackage{graphicx}
\usepackage[T2A]{fontenc}
\usepackage[utf8]{inputenc}
\usepackage[english]{babel}
\usepackage{hyphenat}
\usepackage{mathtools}
\usepackage{amsfonts,amssymb,mathrsfs,amscd,amsmath,amsthm}
\usepackage{indentfirst}
\usepackage{verbatim}
\usepackage{url}
\usepackage{hyperref}
\usepackage{stmaryrd}
\usepackage{alltt}
\graphicspath{{pictures/}}
\DeclareGraphicsExtensions{.pdf,.png,.jpg}
\usepackage{caption}

\theoremstyle{plain}
\newtheorem{theorem}{Theorem}[section]
\newtheorem{pred}[theorem]{Proposition}

\theoremstyle{definition}

\newtheorem{defin}[theorem]{Definition}
\newtheorem{remark}[theorem]{Remark}

\newcommand{\ft}{\operatorname{ft}}
\newcommand{\dist}{\operatorname{dist}}

\def\ig#1#2#3#4{\begin{figure}[!ht]\begin{center}%
\includegraphics[height=#2\textheight]{pictures//#1.png}\caption{#4}\label{#3}%
\end{center}\end{figure}}
\title{The Fermat--Torricelli problem in the case of three-point sets in normed planes}
\author{Daniil A. Ilyukhin}
\date{}

\begin{document}

\maketitle

\begin{abstract}
In the paper the Fermat--Torricelli problem is considered. The problem asks a point minimizing the sum of distances to arbitrarily given points in \emph{d}-dimensional real normed spaces. Various generalizations of this problem are outlined, current methods of solving and some recent results in this area are presented. The aim of the article is to find an answer to the following question: in what norms on the plane is the solution of the Fermat--Torricelli problem unique for any three points. The uniqueness criterion is formulated and proved in the work, in addition, the application of the criterion on the norms set by regular polygons, the so-called lambda planes, is shown.
\end{abstract}

%%%%%%%%%%%%%%%%%%%%%%%%%%%%%%
\section{Introduction}
%%%%%%%%%%%%%%%%%%%%%%%%%%%%%%

The problem of finding a point that minimizes the sum of distances from it to a given set of points in a metric space was first mentioned in the 17th century. In 1643 Fermat posed a problem for three points on the Euclidean plane, and in the same century Torricelli proposed a solution to this problem (\cite{History}).

\ig{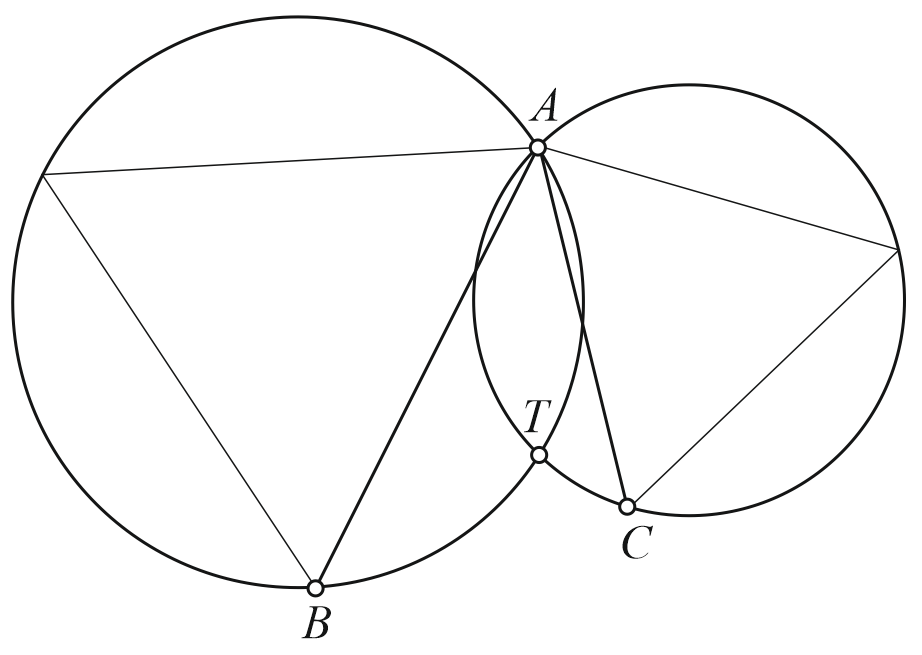}{0.32}{fig:pic4}{The design proposed by Torricelli. The point $T$ is the solution of the problem for the given points $A,B,C$ (\cite{Torricelli})}

Since then, various generalizations of this problem have been considered. The problem was formulated for an arbitrary number of points, the dimension of the space, as well as the norm given in this space. The simplicity of the formulation allows us to consider the problem even in an arbitrary metric space. For example, the problem for four points on the Euclidean plane was solved by D. Fagnano (\cite{bib4}, \cite{bib5}). And for the case of five points, it was proved that the problem is unsolvable in radicals, the proof is given in \cite{bib6} and \cite{bib7}. In addition, there is a generalized problem in which the vertices are considered together with some positive values, called weights. You can read about the development of the weighted problem in the works \cite{bib8}, \cite{arhiv1}, \cite{arhiv2}. In particular, the existence and uniqueness of the solution of such a problem for three points on the Euclidean plane were proved (\cite{bib5}).

In 1934 Jarnik and Kessler \cite{bib9} posed the problem of finding a graph of minimum length connecting a finite number of points in the Euclidean plane, and in 1941 it was discussed in Courant and Robbins's book ``What is Mathematics?`` (\cite {What}), the authors of which named the problem under consideration \emph{Steiner's problem} in honor of Jakob Steiner. The theory of the Steiner problem is presented in the monographs \cite{Extreme1} and \cite{Extreme2}, and in the book ``Theory of Extreme Networks`` (\cite{Extreme3}) Ivanov and Tuzhilin proposed various generalizations of the problem and new lines of research. You can read about the study of extremal networks in normed spaces in \cite{bib12}, \cite{bib13}.

This article will consider the classic version of the problem: finding a point for which the minimum sum of distances to elements of a subset of a metric space is reached. We will call such a formulation \emph{generalized Fermat--Torricelli problem} (or simply \emph{Fermat--Torricelli problem}). The work is based on the article \cite{FTproblem}, which describes the application of a geometric approach to the problem and presents some new results that are obtained in the framework of real finite-dimensional normed spaces, called \emph{Minkowski spaces}. In particular, the authors give the following result:

\begin{theorem}[\cite{FTproblem}]
In Minkowski space, the solution of the Fermat--Torricelli problem is a singleton for any finite set of non-collinear points if and only if the norm is strictly convex.
\end{theorem}

The purpose of this work is to solve a problem with a weaker condition: we will look for norms on the plane in which the solution is unique for any three-point boundaries. In this case, new examples of norms appear that have the property of uniqueness of the solution for any set, for example, the Manhattan plane.

The second section presents the main definitions, and also describes the geometric method for finding a solution to the Fermat--Torricelli problem. Section 3 gives an answer to the question posed: a proof of the uniqueness criterion is given. The last part is devoted to applying the resulting criterion to the norms given by regular polygons.

I would like to express my gratitude to my scientific adviser, Doctor of Physical and Mathematical Sciences Professor A.A.Tuzhilin and Doctor of Physical and Mathematical Sciences Professor A.O.Ivanov for posing the problem and constant attention to the work.

%%%%%%%%%%%%%%%%%%%%%%%%%%%%%%
\section{Basic definitions and preliminary results}
%%%%%%%%%%%%%%%%%%%%%%%%%%%%%%

All statements in this section will be formulated for the Minkowski space, so this will not be specified.

\begin{defin}
A point $x_0$ is called a \emph{Fermat--Torricelli point} for points $A=\{x_1,\ldots,x_n\}$ if $x=x_0$ minimizes $\sum_{i=1}^n| xx_i|$. The set of all such points will be denoted by $\ft(A)$.
\end{defin}

From the properties of the function $\sum_{i=1}^n|xx_i|$, one can obtain the following assertion about the set of solutions of the Fermat--Torricelli problem, see for example \cite{bib10}.

\begin{pred}
Let $A = \{x_1$,\ldots, $x_n\}$ be points in space. Then $\ft(A)$ is a non-empty, compact and convex set.
\end{pred}

We give two examples of solving the Fermat--Torricelli problem on a plane. The figure \ref{fig:pic5} shows the vertices of an equilateral triangle, first in the Euclidean plane, and then in the norm given by a regular hexagon. In the first case, the set of solutions contains a single point constructed in such a way that the angles between the rays coming out of it in the direction of the vertices of the triangle are equal. In the second case, we specify the location of the points of the given set: let one of them be at the origin, and the other two --- at neighboring vertices of the unit circle. Under such conditions, the set of solutions will include all points of the constructed triangle, including the interior and boundaries.

\ig{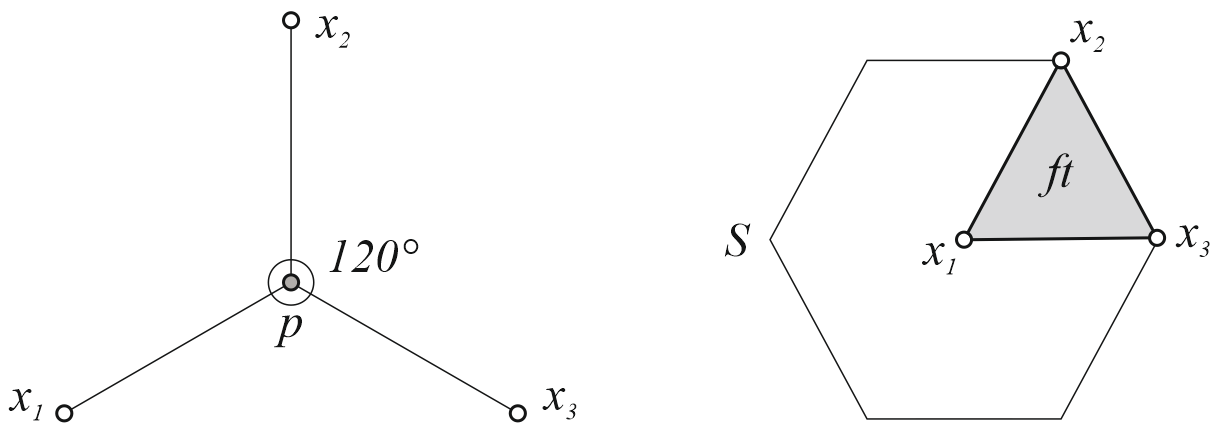}{0.21}{fig:pic5}{Examples of solutions to the Fermat--Torricelli problem on the Euclidean plane and on the $\lambda$-plane}

Now the geometric method for constructing the solution of the Fermat--Torricelli problem will be presented. To use it, along with the original space $X$, one must consider its dual space $X^*$, which consists of linear functionals.

\begin{defin}
A functional $\varphi\in X^*$ is called \emph{norming for a vector} $x\in X$ if $\|\varphi\|=1$ and $\varphi(x)=\|x\|$ .
\end{defin}

It is easy to see that the elements of the space $X$ obtained by multiplying the vector $x$ by a positive number $k$ have the same set of norming functionals. That is, the set of such functionals can be described using points of the unit sphere. Consider norming functionals on some normed plane.

The figure \ref{fig:pic6} shows a section of the unit circle $S$ containing both a flattened and a smooth section. Let us construct norming functionals for vectors starting at zero and ending at a point lying on $S$. At the point $z$ the unit circle has a single support line, then for the corresponding vector there is a unique norming functional $\varphi_4$, its level line coincides with this support line. Internal flattening points $xy$ also correspond to vectors with a unique norming functional. At the point $x$ the unit circle has more than one reference line. Each of them defines a norming functional, that is, for a vector with an end in $x$, there are infinitely many norming functionals. The functionals $\varphi_1$ and $\varphi_2$ are given by the limit positions of the reference line, while $\varphi_3$ is an arbitrary one.

\ig{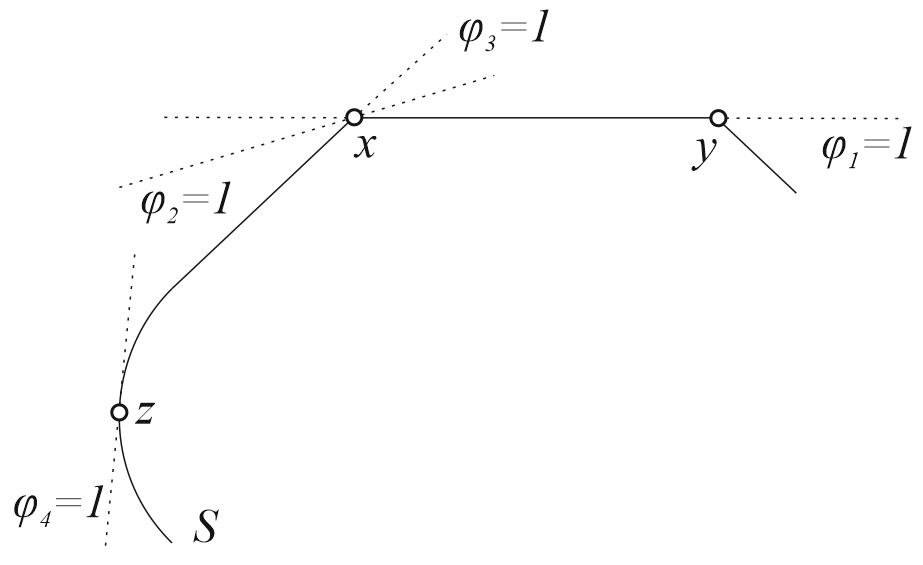}{0.25}{fig:pic6}{Construction of norming functionals along the support lines to the unit circle}

The following theorem is a criterion for a certain point to belong to the set of solutions of the Fermat--Torricelli problem.

\begin{theorem}[\cite{bib11}, \cite{Extreme4}]\label{thm:ftpoint}
Let $x_0$, $x_1$, ..., $x_n$ be points in space and $x_0 \neq x_i$ for $i=1,...,n$. Then $x_0$ is a Fermat--Torricelli point for $A=\{x_1,...,x_n\}$ if and only if each vector $x_i-x_0$, $i=1,.. .,n$, has a norming functional $\varphi_i$ such that $\sum_{i=1}^n\varphi_i=0$.
\end{theorem}

Let us give the simplest example of using this theorem on the Euclidean plane. Let's consider three points on a circle, located at an equal distance from each other, and suppose that the origin of coordinates is the solution of the Fermat--Torricelli problem for them. For each of the points $x_1,x_2,x_3$ there is a unique support line, which is the level line $\varphi_i=1$ of some functional. The three constructed lines form an equilateral triangle and are equidistant from the origin, therefore, the sum of the functionals they define is equal to zero. By the \ref{thm:ftpoint} theorem, the point $p=0$ belongs to the set of solutions.

\ig{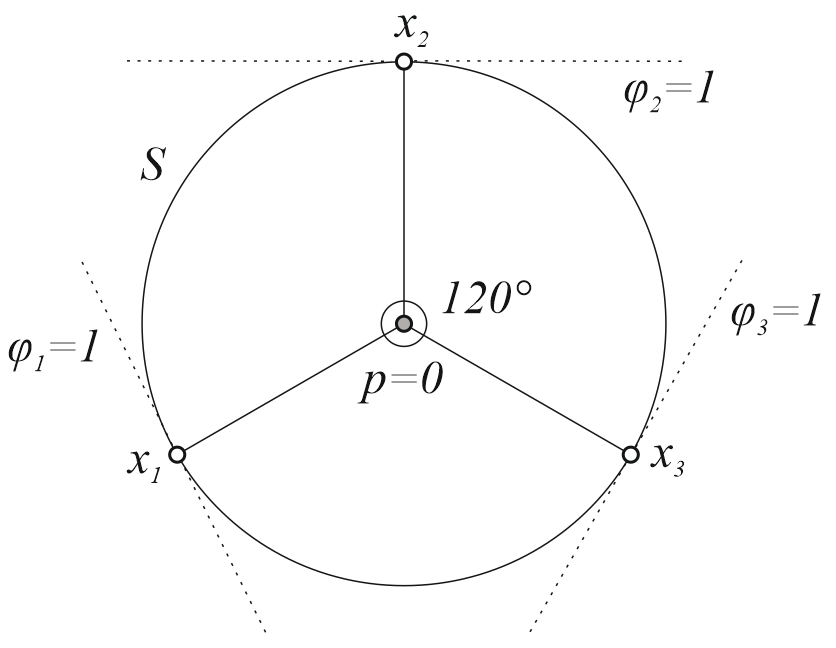}{0.27}{fig:pic8}{The origin of coordinates belongs to the solution set for points $x_1,x_2,x_3$}

Assume that a point is found that is a solution for the set $A$. Now, using the functionals from the \ref{thm:ftpoint} theorem, we can construct the entire set $\ft(A)$. To do this, we introduce a new object.

\begin{defin}
Let a functional $\varphi\in X^*$ and a point $x\in X$ be given. Define \emph{cone} $C(x,\;\varphi)=x-\bigl\{ a:\;\varphi(a)=\|a\|\bigr\}$.
\end{defin}

Let's consider two examples of constructing a cone on the Manhattan plane (figure \ref{fig:pic9}). Let the functional $\varphi_1$ be given by the support line intersecting the unit circle at one point, and $\varphi_2$ be the line containing the flattening. The set $\bigl\{ a:\;\varphi_1(a)=\|a\|\bigr\}$ is a ray that leaves the origin and passes through this point. For the second functional, this will be a whole set of rays intersecting all flattening points. In both cases, we then reflect the constructed set relative to the origin of coordinates and shift it by parallel translation so that the vertex hits the given point.

\ig{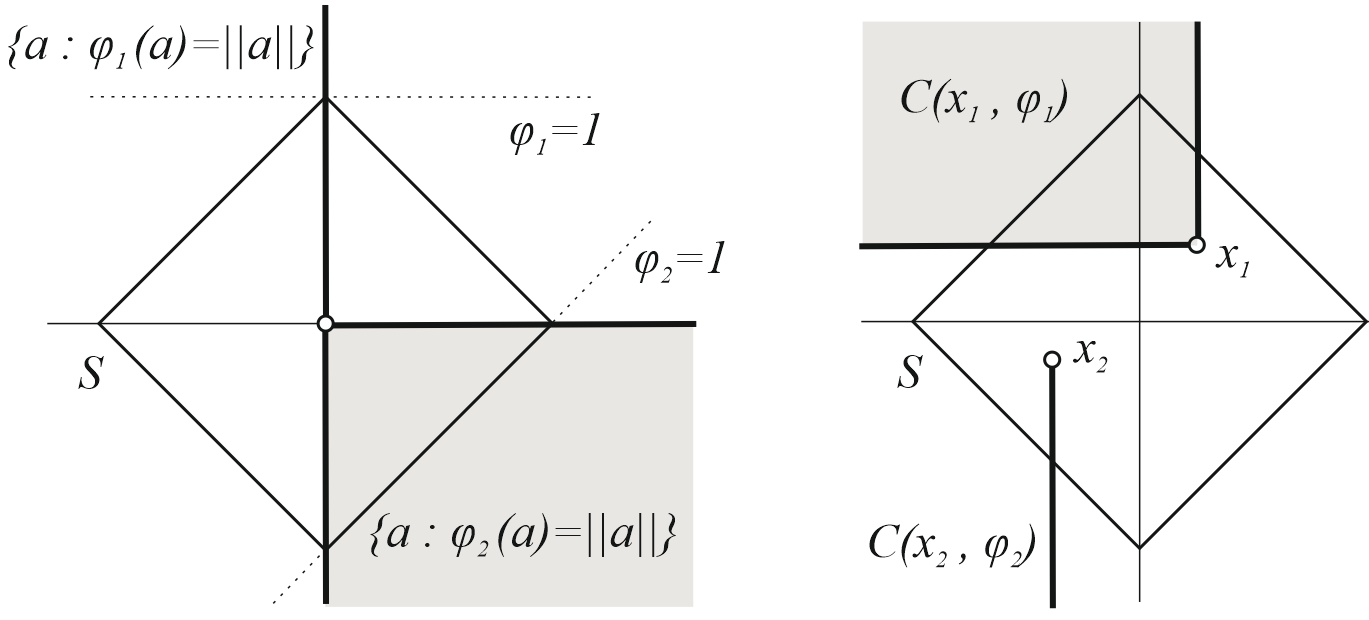}{0.28}{fig:pic9}{Construction of cones on the Manhattan plane}

\begin{theorem}[\cite{bib11}]\label{thm:ftlocus}
Let $A = \{x_1$, ..., $x_n\}$ be points in space and $p\in \ft(A)\setminus A$. By the \ref{thm:ftpoint} theorem, for each vector $x_i-p$, $i=1,...,n$, there exists a norming functional $\varphi_i$ such that $\sum_{i=1}^n \varphi_i=0$. Then $\ft(A)=\cap_{i=1}^n C(x_i,\;\varphi_i)$.
\end{theorem}

The \ref{thm:ftpoint} and \ref{thm:ftlocus} theorems constitute a geometric method for finding a solution to the Fermat--Torricelli problem. To describe the application of this theorem, let's consider in more detail the example with the vertices of an equilateral triangle in the hexagonal norm (figure \ref{fig:pic7}).

First, let's find at least one solution. Take $p=\frac{1}{3}(x_1+x_2+x_3)$ and use the \ref{thm:ftpoint} theorem to prove that $p\in\ft(A)$. To do this, consider the vectors $x_1-p,x_2-p,x_3-p$ and construct their norming functionals. Since they lie in the same directions with internal flattening points, then for each of the vectors $x_i-p$ there is exactly one functional $\varphi_i$ whose level line contains the corresponding flattening. The flattenings are equidistant from the origin and form an equilateral triangle; therefore, the sum of the functionals constructed from them is equal to zero, and the condition of the theorem is satisfied.

\ig{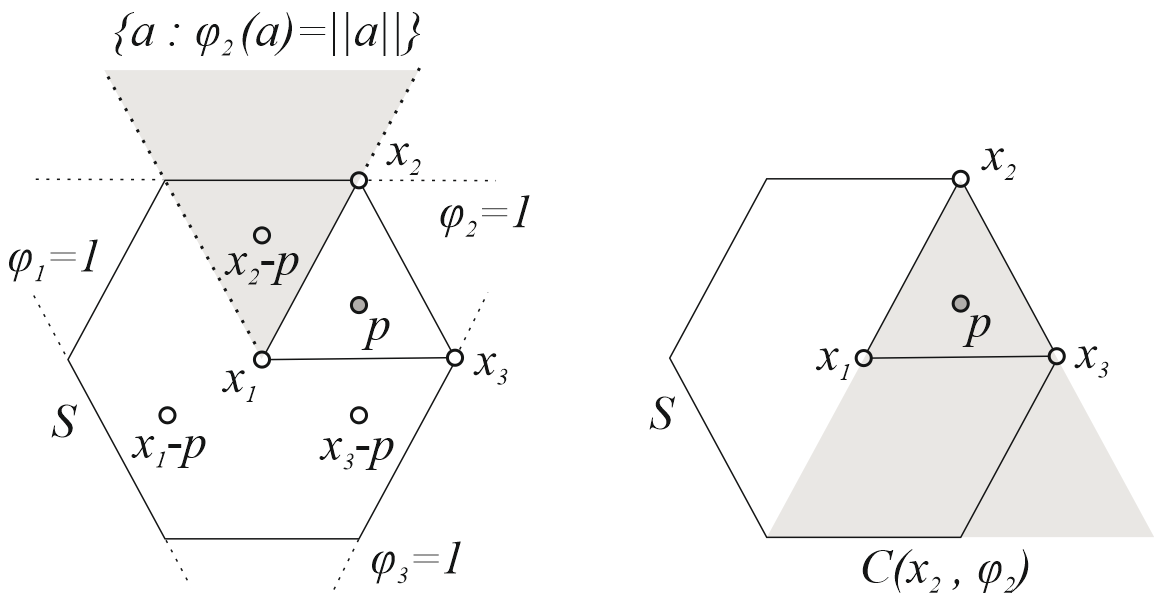}{0.27}{fig:pic7}{Construction of a complete set of solutions for points $x_1,x_2,x_3$ on a plane with hexagonal norm}

Now let's use the \ref{thm:ftlocus} theorem to find all solutions. Let us construct a cone given by the functional $\varphi_2$ and the vector $x_2-p$. Since the functional is norming for all flattening points, the set $\bigl\{ a:\;\varphi_2(a)=\|a\|\bigr\}$ is a set of rays emanating from the origin and passing through the flattening points , that is, an angle whose sides contain two adjacent vertices. Now we will reflect the angle relative to the origin and perform a parallel translation so that the vertex of the angle is at the point $x_2$. After a similar construction of two other cones, we obtain that their intersection is the entire triangle $x_1x_2x_3$.

\begin{remark}
The statement of \ref{thm:ftlocus} theorem does not depend on the choice of the point $p$ and the functionals $\varphi_{i}$.
\end{remark}

The geometric method gives a general description of the solutions of the Fermat--Torricelli problem for various given sets - this is the intersection of some cones with vertices located at the points of this set. In the case of a plane, this observation allows us to formulate the following statements:

\begin{pred}[\cite{FTproblem}]\label{thm:fttypes}
Let $A = \{x_1$, ..., $x_n\}$ be points on the plane. Then $\ft(A)$ is a convex polygon that can degenerate into a segment or a point.
\end{pred}

\begin{pred}[\cite{FTproblem}]\label{thm:oddset}
Let in space the points of the set $A = \{x_1$, ..., $x_{2k+1}\}$ be located on one straight line in the order of their numbering. Then $\ft(A)=\{x_{k+1}\}$.
\end{pred}

%%%%%%%%%%%%%%%%%%%%%%%%%%%%%%
\section{Uniqueness criterion}
%%%%%%%%%%%%%%%%%%%%%%%%%%%%%%

Let a set of three points $A = \{x_1,x_2,x_3\}$ be given on a normed plane.
Answering the question what form the set of solutions of the Fermat--Torricelli $\ft(A)$ problem has under these conditions, we can get three cases: a non-degenerate convex polygon, a segment, and a point.
To formulate a uniqueness criterion, we solve the inverse problem: find the necessary and sufficient conditions for $\ft(A)$ to be a polygon or a segment.

The \ref{thm:ftlocus} theorem says that the set $\ft(A)$ is always the intersection of three cones that are angles or rays on the plane, i.e. $$\ft(A)=C(x_1 ,\;\varphi_1)\cap C(x_2,\;\varphi_2)\cap C(x_3,\;\varphi_3)$$ for some $\varphi_i$ functionals.

\ig{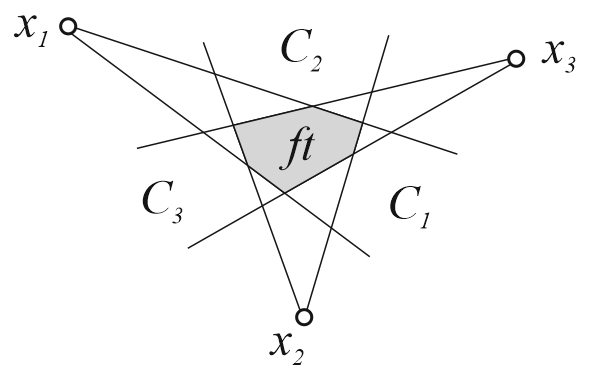}{0.18}{fig:pic1}{The solution of the Fermat--Torricelli problem is the intersection of cones emerging from points of a given set}

\begin{pred}\label{thm:polygon}
If $\ft(A)$ is a non-degenerate polygon, then the cones $C_i$ are non-degenerate angles.
\end{pred}

\begin{proof}
In the case when at least one of the cones is a ray, the set $\ft(A)$ is a subset of this ray and cannot be a non-degenerate polygon.
\end{proof}

\begin{pred}\label{thm:segment}
If $\ft(A)$ is a non-degenerate segment, then at least one of the cones $C_i$ is a non-degenerate angle.
\end{pred}

\begin{proof}
If all three cones are rays, then their intersection can be a segment only if they all lie on the same straight line. In this case the points $x_i$ lie on the same line, and by the proposition \ref{thm:oddset} the set $\ft(A)$ coincides with one of these points. Contradiction.
\end{proof}

The statements \ref{thm:polygon} and \ref{thm:segment} formulate the necessary conditions, now let's move on to the conditions on the norm, which ensure the existence of a set of points with a non-unique solution.

\begin{defin}
We will say that the unit circle on the normed plane consists of \emph{elements of zero type}, which are points that do not lie inside flattenings, and \emph{elements of the first type}, which are internal points of flattenings.
\end{defin}

\begin{defin}\label{thm:good3}
Take three elements of the unit circle. For each of them, we choose a support line so that for an element of the zero type it intersects the unit circle only at one point, that is, in the element itself. Using the chosen support lines, we construct functionals whose level lines $\varphi_i=1$ coincide with these lines. If there is a set of support lines such that the sum of the constructed functionals is equal to zero, then we will call such a triple \emph{consistent}.
\end{defin}

\begin{pred}\label{thm:1cond}
Let some norm be given on the plane. If there is a consistent triple of elements of the unit circle consisting of three elements of the first type, then in this plane there are three points $x_1,x_2,x_3$ for which the set of solutions $\ft(x_1,x_2,x_3)$ is a non-degenerate polygon .
\end{pred}

\begin{proof}
Let $x_1,x_2,x_3$ be elements of this triple. Then, taking the functionals $\varphi_1, \varphi_2, \varphi_3$ from the definition \ref{thm:good3}, by the theorem \ref{thm:ftpoint} we get that $p=0$ belongs to $\ft(x_1, x_2, x_3 )$. By the \ref{thm:ftlocus} theorem, $\ft(x_1, x_2, x_3)$ is the intersection of the cones coming out of the points $x_1, x_2, x_3$. Since all functionals contain flattenings, all three cones are non-degenerate angles. Moreover, each of them contains some neighborhood of zero due to the fact that the points $x_1, x_2, x_3$ are not flattening boundaries. Therefore, the cones at the intersection form a polygon.
\end{proof}

\ig{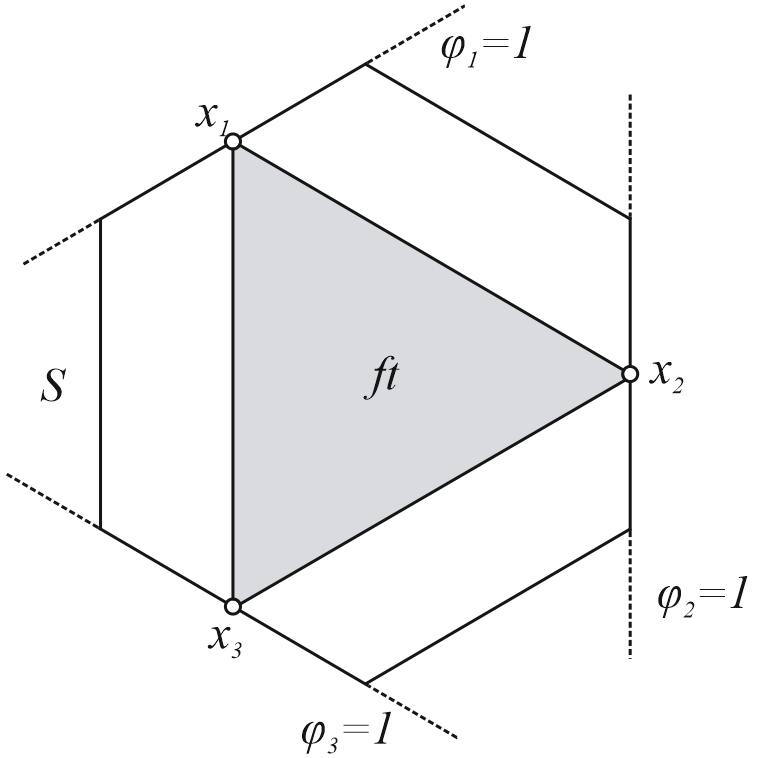}{0.32}{fig:pic2}{The first non-uniqueness condition}

\begin{pred}\label{thm:2cond}
Let some norm be given on the plane. If there is a consistent triple of elements of the unit circle, consisting of two elements of the first type and one element of the zero type, then in this plane there are three points $x_1,x_2,x_3$ for which the set of solutions $\ft(x_1,x_2,x_3)$ --- non-degenerate segment.
\end{pred}

\begin{proof}
Let the points $x_1,x_2$ be elements of the first type, and $x_3$ be the zero type from the given triple. Similarly to the previous proof, we get that $p=0$ belongs to $\ft(x_1,x_2,x_3)$. The cones emerging from $x_1,x_2$ contain a neighborhood of zero, and a ray passes from the point $x_3$ and passes through the origin. Consequently, at the intersection we obtain a non-degenerate segment.
\end{proof}

The following condition ensures the existence in a given norm of three points for which the set of solutions is the intersection of one non-degenerate angle and two rays.

\begin{pred}\label{thm:3cond}
Let some norm be given on the plane. If there is a consistent triple of elements of the unit circle, consisting of one element of the first type and two elements of the zero type, and the elements of the zero type are located symmetrically with respect to the origin, then in this plane there are three points $x_1,x_2,x_3$, for which the set of solutions $ \ft(x_1,x_2,x_3)$ is a non-degenerate segment.
\end{pred}

\begin{proof}
Let the points $x_1,x_2$ be elements of the zero type, and $x_3$ the first type from the given triple. The functionals $\varphi_1,\varphi_2,\varphi_3$ from the definition of \ref{thm:good3} satisfy the \ref{thm:ftpoint} theorem, and the point $x=0$ belongs to $\ft(x_1,x_2,x_3)$. By the \ref{thm:ftlocus} theorem, the solution $\ft(x_1,x_2,x_3)$ is the intersection of two rays lying on the same straight line containing elements of type zero and a non-degenerate angle coming out of the flattening. We get a segment.
\end{proof}

\ig{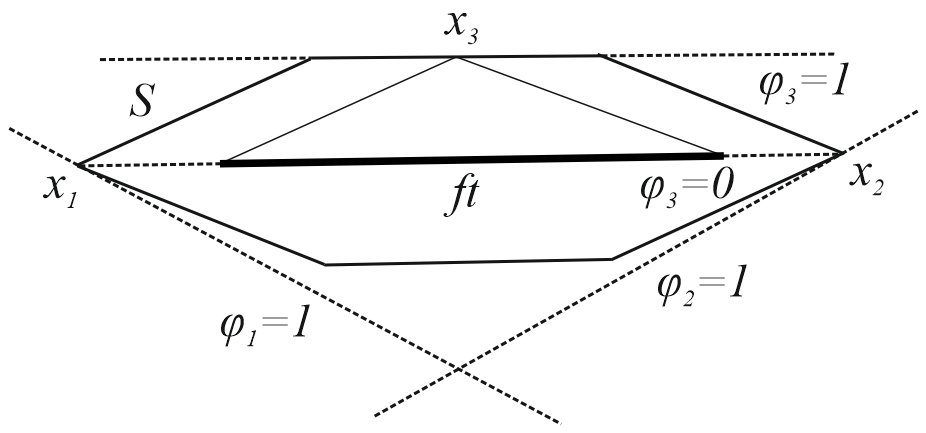}{0.21}{fig:pic3}{The third non-uniqueness condition}

\begin{defin}
We will say that \emph{the first, second or third non-uniqueness condition} is satisfied for a norm on the plane if it satisfies the condition of the proposition \ref{thm:1cond}, \ref{thm:2cond} or \ref{thm:3cond} respectively .
\end{defin}

\begin{theorem}\label{thm:criteria}
The solution of the Fermat--Torricelli problem is unique for any three points in a normed plane if and only if none of the non-uniqueness conditions is satisfied for the given norm.
\end{theorem}

\begin{proof}
If the norm satisfies at least one of the non-uniqueness conditions, then, by the corresponding assertion, there are three points on the plane for which the set of solutions is infinite.

Let there be three points $x_1,x_2,x_3$ for which $\ft(x_1,x_2,x_3)$ is not unique. Since the solution is the intersection of cones, it is either a polygon or a segment. In the case of a polygon, all three cones are nondegenerate angles. Consider the functionals defining the cones. Their level lines contain flattenings of the unit circle. Taking one interior point from each flattening, we obtain a consistent triple of elements of the first type, which is required in the first non-uniqueness condition.

A segment can be obtained in three ways: by the intersection of three angles, by the intersection of two angles and a ray, and by the intersection of one angle and two rays. In the first two cases, we similarly consider the functionals defining the cones. We obtain the first and second non-uniqueness conditions, respectively. Now let's turn to the third case. First, we obtain a consistent triple consisting of one flattening and two points. Secondly, in order for the solution to be non-unique, the rays must lie on the same straight line, that is, each of the two rays must contain the beginning of the other. Hence it follows that the elements of the zero type from the matched triple are points lying opposite each other with respect to the origin. Consequently, the third non-uniqueness condition is satisfied for this norm.
\end{proof}

%%%%%%%%%%%%%%%%%%%%%%%%%%%%%%
\section{Lambda-normed planes}
%%%%%%%%%%%%%%%%%%%%%%%%%%%%%%

In practice, the \ref{thm:criteria} criterion is not easy to apply, since the answer for a given norm is given by enumerating the flattenings of the unit circle. If we consider it as an algorithm, then its complexity directly depends on their number, and due to the presence of norms with an infinite number of flattenings, it is not even computable in finite time. Let us show its application to the most popular family of nonstrictly convex norms, the $\lambda$-planes. Extremal networks in $\lambda$-normed planes are studied by Ilyutko in \cite{bib15}, \cite{bib14}.

\begin{defin}
A $\lambda$\emph{-normed plane} is a plane whose norm is given by a regular $2\lambda$-gon.
\end{defin}

\begin{remark}
We use the notation $\dist(x, y)$ for the Euclidean distance between points $x$ and $y$ on the plane.
\end{remark}

\begin{pred}
If $\lambda\not\equiv0\mod{3}$, a $\lambda$-normed plane does not satisfy the first non-uniqueness condition.
\end{pred}

\begin{proof}
Assume that the condition is satisfied and there are functionals $\varphi_1,\varphi_2,\varphi_3$ such that their sum is equal to zero and the lines $\varphi_i=1$ contain flattenings of the unit circle.

Let $l_i$ be straight lines defining the level lines $\varphi_i=1$. The sum of the vectors defining these lines is equal to zero, as well as $\dist(0,\;l_1)=\dist(0,\;l_2)=\dist(0,\;l_3)$, since $l_i$ contain flattening of a $2\lambda$-gon. Therefore, the level lines form a regular triangle, which is impossible for $\lambda\not\equiv0\mod{3}$.
\end{proof}

\begin{pred}
If $\lambda\not\equiv0\mod{3}$, a $\lambda$-normed plane does not satisfy the second non-uniqueness condition.
\end{pred}

\begin{proof}
Let us assume that the condition is satisfied and there are functionals $\varphi_1,\varphi_2,\varphi_3$ such that their sum is equal to zero, the lines $\varphi_1=1$ and $\varphi_2=1$ contain flattenings of the unit circle, and the line $\varphi_3 =1$ --- vertex.

Let $l_i$ be straight lines defining the level lines $\varphi_i=1$. Let's rotate the plane so that the line $l_3$ is parallel to the coordinate axis $x$. Let $l_1$ and $l_2$ be given by the equations $\alpha_1x+\beta_1y=1$ and $\alpha_2x+\beta_2y=1$, and $l_3$ be given by the equation $\gamma y=1$. Since $\sum_{i=1}^3\varphi_i=0$, then $\alpha_1+\alpha_2=0$ and $\beta_1+\beta_2+\gamma=0$. Since the flattenings of the polygon belong to the lines $l_1$ and $l_2$, we obtain $$\dist(0,\;l_1)=\dist(0,\;l_2)\Rightarrow\alpha_1^2+\beta_1^2= \alpha_2^2+\beta_2^2.$$
Note that the lines $l_1$ and $l_2$ are not parallel. Otherwise, the sum of the functionals $\varphi_1$ and $\varphi_2$ is equal to zero, since $l_1$ and $l_2$ are at the same distance from the origin. This means that the functional $\varphi_3$ is zero, but it is not, because its level line is the reference line to the unit circle. Considering this,
$$-\alpha_1=\alpha_2=\alpha,\;\beta_1=\beta_2=\beta,\;\gamma=-2\beta.$$
Thus, the equations for lines $l_i$ have taken the form:
$$-\alpha x+\beta y=1,\;\alpha x+\beta y=1,\;-2\beta y=1.$$
The triangle formed by the lines $l_i$ is isosceles, its base is parallel to the coordinate axis $x$. Let it be located below the axis, that is, let $\beta>0$. Let's also assume that $\alpha>0$, otherwise replace $\alpha$ with $-\alpha$. Let's draw a bisector from the point of intersection of the lines $l_1$ and $l_2$. The bisector lies on the y-axis. Two cases are possible: either it passes through the midpoints of opposite sides of the $2\lambda$-gon, or through a pair of opposite vertices. The first case contradicts the fact that the support line $l_3$ does not contain flattening. Then the bisector passes through the vertex touched by $l_3$. Since $l_3$ is parallel to the x-axis, we obtain the equality of the angles formed by the line $l_3$ and the sides of the $2\lambda$-gon having a common point with it.

We impose one more condition on the coefficients $\alpha$ and $\beta$. Let $n$ be the number of polygon vertices. We have
$$\dist(0,\;l_1)=\dist(0,\;l_2)=\cos{\frac{\pi}{n}}\dist(0,\;l_3)
\Rightarrow
\frac{1}{\sqrt{\alpha^2+\beta^2}}=\frac{\cos{\frac{\pi}{n}}}{2\beta}
\Rightarrow
\alpha^2+\beta^2=\frac{4}{\cos^2{\frac{\pi}{n}}}\beta^2.$$ 
Hence 
$$\frac{\alpha}{\beta}=\sqrt{\frac{4}{\cos^2{\frac{\pi}{n}}}-1}=\sqrt{4\tan^2{\frac{\pi}{n}}+3}.$$ 

If $\xi_{12}, \xi_{13}, \xi_{23}$ are the angles between the lines $l_1, l_2, l_3$, then $\xi_{13}=\xi_{23}$ and $\xi_{12}+\xi_{13}+\xi_{23}=\pi$. Since $\tan{\xi_{13}}=\tan{\xi_{23}}=\frac{\alpha}{\beta}$, then $$\xi_{12}=\pi-2\arctan{\tan{\xi_{13}}}=\pi-2\arctan{\sqrt{4\tan^2{\frac{\pi}{n}}+3}}.$$ 
On the other hand, $\xi_{12}=\pi-\frac{2\pi k}{n}$ for some $k$, $k\in\mathbb {N}$ $1\le k<\frac{ n}{2}$. Then 
$$\frac{\pi k}{n}=\arctan{\sqrt{4\tan^2{\frac{\pi}{n}}+3}}.$$
Let us prove that this equality does not hold for natural values of $k$ and $n$ such that $n\equiv\pm1\mod{3}$, and come to a contradiction with the original assumption. To prove this fact, it suffices to verify that the following double inequality is valid for all $n>2$:
$$\frac{n}{3}-\frac{1}{3}<\frac{n}{\pi}\arctan{\sqrt{4\tan^2{\frac{\pi}{n}}+3}}<\frac{n}{3}+\frac{1}{3}$$ 
First, we multiply the inequality by $\frac{\pi}{n}$:
$$\frac{\pi}{3}-\frac{\pi}{3n}<\arctan{\sqrt{4\tan^2{\frac{\pi}{n}}+3}}<\frac{\pi}{3}+\frac{\pi}{3n}.$$
The left side of the inequality is less than $\frac{\pi}{3}$ for all $n>2$, and the middle part is greater than $\frac{\pi}{3}$ for the same $n$, so the first inequality holds . In the remaining inequality, we take the tangent of both parts and square it:
$$4\tan^2{\frac{\pi}{n}}+3<\tan^2{\Big(\frac{\pi}{3}+\frac{\pi}{3n}\Big)}.$$
Next, we make the substitution $x=\tan{\frac{\pi}{3n}}$, and after a few simple transformations we get the inequality
$$4\Big(\frac{3x-x^3}{1-3x^2}\Big)^2+3<\Big(\frac{x+\sqrt{3}}{1-\sqrt{3}x}\Big)^2,$$
which must be carried out for $0<x<\frac{1}{\sqrt{3}}$. It is easy to see that this is indeed the case. The assertion has been proven.
\end{proof}

\begin{pred}
A $\lambda$-normed plane does not satisfy the third non-uniqueness condition.
\end{pred}

\begin{proof} 
Let's assume that the condition is met. Then there is a compatible triple consisting of an internal point of flattening and a pair of vertices lying symmetrically about the center of the polygon. Let $\varphi$ be the functional corresponding to the support line to the element of the first type. Its level line $\varphi=1$ contains a flattening of the polygon, and the line $\varphi=0$ passes through elements of type zero, since the sum of the functionals is equal to zero.
Thus, the norms given by $4n$-gons are no longer suitable, because one of the flattenings is parallel to the symmetry axis of the polygon. Next, let $n\geq6$, consider the support lines corresponding to elements of zero type. Their intersection point must belong to the line $\varphi=-2$. The angle between the line $\varphi=0$ and the adjacent flattening must be less than the angle between the line $\varphi=0$ and the reference line. Let's move on to the inequality of their tangents: $$\tan{\frac{\pi(n-2)}{2n}}<2\cos{\frac{\pi}{n}}$$
This inequality does not hold for $n\geq6$. That is, we obtain a contradiction for all $\lambda$-planes.
\end{proof}

\begin{defin}
Consider a triangle $ABC$ in the plane, all angles of which are less than $120^{\circ}$. Then $T$ is the \emph{Torricelli point} for the vertices $A,B,C$ if the angles $ATB,BTC,ATC$ are equal to $120^{\circ}$.
\end{defin}

\begin{remark}
We use the notation $\widehat{x}$ for $\frac{1}{\|x\|}x$ for $x\not=0$.
\end{remark}

\begin{pred}
Let $A = \{x_1, x_2, x_3\}$ be the set of points on the $\lambda$-normed plane that do not lie on one straight line, and $\lambda\equiv0\mod{3}$. Let all angles in a triangle with vertices $x_1, x_2, x_3$ be less than $120^{\circ}$, and $p$ be the Torricelli point for this triangle. Then if the points of $\widehat{x_i-p}$ do not coincide with the vertices of the $2\lambda$-gon, then $\ft(A)$ is a non-degenerate convex polygon, otherwise $\ft(A)=\{p\}$ .
\end{pred}

\begin{proof}
Let us construct the norming functionals $\varphi_1,\;\varphi_2,\;\varphi_3$ for the vectors $x_1-p,\;x_2-p,\;x_3-p$. If $\widehat{x_i-p}$ belong to flattenings of the unit circle, then we put $\varphi_i=1$ on this flattening. If $\widehat{x_i-p}$ are vertices of a polygon, then we choose flattenings in the same traversal direction. Then $\varphi=\varphi_1+\varphi_2+\varphi_3=0$ and by Theorem \ref{thm:ftpoint} $p\in \ft(A)$. By Theorem \ref{thm:ftlocus} $\ft(A)=\cap_{i=1}^3 C(x_i,\;\varphi_i)$. In the first case, each cone $C(x_i,\;\varphi_i)$ together with $p$ contains some neighborhood $p$, and $\ft(A)$ is a non-degenerate convex polygon, otherwise the point $p$ lies on the boundary of each of the cones $C(x_i,\;\varphi_i)$ and due to the symmetry $\cap_{i=1}^3 C(x_i,\;\varphi_i)=\{p\}$
\end{proof}

The assertions proved above imply a theorem classifying the norms of a given family with respect to the uniqueness of the solution of the Fermat--Torricelli problem.

\begin{theorem}
On a $\lambda$-normed plane, the solution of the Fermat--Torricelli problem is unique for any three points if and only if $\lambda\not\equiv0\mod{3}$.
\end{theorem}

\label{end}

\end{document}